\documentclass[11pt, leqno]{amsart}
\usepackage{amsfonts,amssymb, amscd}
\usepackage{amsmath}
\usepackage{lmodern}
\usepackage{mathrsfs}
\usepackage{amsthm}
\usepackage{qsymbols}
\usepackage{mathtools}
\mathtoolsset{showonlyrefs}
\usepackage{dsfont}
\usepackage{graphicx}
\usepackage{latexsym}
\usepackage{chngcntr}
\usepackage[noadjust]{cite}
\usepackage{paralist}
\usepackage[parfill]{parskip}
\usepackage{bm}
\usepackage{tikz-cd}
\usepackage{nicefrac}
\usepackage{float}
\usepackage{csquotes}

\usepackage[shortlabels]{enumitem}
\setlist[enumerate,1]{label = \normalfont(\roman*), ref = (\roman*)}
\newtheorem{theorem}{Theorem}[section]
\newtheorem{lemma}[theorem]{Lemma}
\newtheorem{proposition}[theorem]{Proposition}
\newtheorem{corollary}[theorem]{Corollary}
\newtheorem{taggedtheoremx}{Assumption}
\newenvironment{assumption}[1]
 {\renewcommand\thetaggedtheoremx{#1}\taggedtheoremx}
 {\endtaggedtheoremx}
\theoremstyle{definition}
\newtheorem{definition}[theorem]{Definition}
\newtheorem{remark}[theorem]{Remark}

\usepackage[plainpages=false,pdfpagelabels,
citecolor=red]{hyperref}
\usepackage{xcolor}
\hypersetup{
 colorlinks,
 linkcolor={cyan!90!black},
 citecolor={magenta},
 urlcolor={green!40!black}
} %
\newcommand{\R}{\mathbb{R}}
\newcommand{\C}{\mathbb{C}}

\newcommand{\Ext}{\mathrm{E}} %
\renewcommand{\L}{\mathrm{L}}
\newcommand{\W}{\mathrm{W}}
\newcommand{\Cont}{\mathrm{C}}
\newcommand{\Lip}{\mathrm{Lip}}
\newcommand{\IW}{\mathbb{W}}
\newcommand{\ID}{\mathbb{D}}
\newcommand{\e}{\mathrm{e}}
\renewcommand{\d}{\,\mathrm{d}}
\newcommand{\dd}[1]{\,\frac{\mathrm{d} #1}{#1}}
\let\rr\r
\renewcommand{\r}{\mathrm{r}}

\newcommand{\eps}{\varepsilon}
\newcommand{\B}{\mathrm{B}}

\renewcommand{\H}{\mathrm{H}}
\newcommand{\cl}[1]{\overline{#1}}
\newcommand{\cT}{\mathcal{T}}
\newcommand{\Sec}{\mathrm{S}}

\newcommand{\cL}{\mathcal{L}}
\newcommand{\cI}{\mathcal{I}}
\newcommand{\cJ}{\mathcal{J}}
\newcommand{\HM}{\mathcal{H}}
\newcommand{\SP}{\, |\,}

\renewcommand\Re{\operatorname{Re}}

\newcommand{\pLB}[1]{\llbracket #1 \rrbracket}
\newcommand{\qLB}[1]{\llparenthesis \hspace{1pt} #1 \hspace{1pt}\rrparenthesis}
\let\ii\i
\DeclareMathOperator{\bd}{\partial \!}
\DeclareMathOperator{\supp}{supp}
\DeclareMathOperator{\dist}{d}
\DeclareMathOperator{\diam}{diam}

\DeclareMathOperator{\dom}{D}

\def\XXint#1#2#3{{\setbox0=\hbox{$#1{#2#3}{%
\int}$ }
\vcenter{\hbox{$#2#3$ }}\kern-.6\wd0}}

\makeatletter
\@namedef{subjclassname@2020}{%
  \textup{2020} Mathematics Subject Classification}
\makeatother
\title[Square root of elliptic systems with mixed boundary conditions II]{$\L^p$-estimates for the square root of elliptic systems with mixed boundary conditions II}
\author{Sebastian Bechtel}
\address{Institut de Math\'{e}matiques de Bordeaux, Universit\'{e} Bordeaux, UMR CNRS 5251, 351 Cours de la Lib\'{e}ration 33405, Talence, France}
\email{sebastian.bechtel@math.u-bordeaux.fr}
\subjclass[2020]{Primary: 35J47, 47A60. Secondary: 42B20, 26D10.}
\date{\today}
\dedicatory{}
\keywords{complex elliptic systems of second order, mixed boundary conditions, Kato square root problem, Calder\'{o}n--Zygmund decomposition for Sobolev functions, Hardy's inequality, Lax--Milgram isomorphism}
\begin{document}
\begin{abstract}

We show $\L^p$-estimates for square roots of second order complex elliptic systems $L$ in divergence form on open sets in $\R^d$ subject to mixed boundary conditions. The underlying set is supposed to be locally uniform near the Neumann boundary part, and the Dirichlet boundary part is Ahlfors--David regular. The lower endpoint for the interval where such estimates are available is characterized by $p$-boundedness properties of the semigroup generated by $-L$, and the upper endpoint by extrapolation properties of the Lax--Milgram isomorphism. Also, we show that the extrapolation range is relatively open in $(1,\infty)$.

\end{abstract}
\maketitle
\section{Introduction and main results}
\label{Sec: Introduction}

Let $L$ be a second order complex elliptic system in divergence form on an open set $O\subseteq \R^d$, $d\geq 2$, formally given by
\begin{align}
	Lu = -\sum_{i,j=1}^d \partial_i (a_{ij} \partial_j u) - \sum_{i=1}^d \partial_i (b_i u) + \sum_{j=1}^d c_j \partial_j u + d u.
\end{align}
The function $u$ takes its values in $\C^m$, where $m \geq 1$ is the size of the system, and the coefficients $a_{ij}, b_i, c_j, d$ are valued in $\cL(\C^m)$ and are only supposed to be bounded, measurable, and elliptic in the sense of~\eqref{label-Garding}. We refer to $d$ and $m$ as \emph{dimensions}, and to the implied constants in boundedness and ellipticity as \emph{coefficient bounds}. The system $L$ is subject to mixed boundary conditions in the following sense: We fix some closed subset $D\subseteq \bd O$, on which we impose homogeneous Dirichlet boundary conditions, and in the complementary boundary part $N\coloneqq \bd O \setminus D$ we impose natural boundary conditions.

To make this more precise, denote by $\W^{1,2}_D(O)$ the first-order Sobolev space on $O$ with a homogeneous Dirichlet boundary condition on $D$. A proper definition will be given in Section~\ref{Subsec: Sobolev spaces}. Put $\IW^{1,2}_D(O) \coloneqq \W^{1,2}_D(O)^m$ and define the sesquilinear form $a\colon \IW^{1,2}_D(O) \times \IW^{1,2}_D(O) \to \C$ by
\begin{align}
\label{a}
 a(u,v) = \int\limits_O \sum_{i,j=1}^d   a_{ij} \partial_j u \cdot \cl{\partial_i v} + \sum_{i=1}^d  b_{i} u \cdot \cl{\partial_i v} + \sum_{j=1}^d   c_{j} \partial_j u \cdot \cl{v} + d u \cdot \cl{v}  \d x.
\end{align}
Define $L$ as the operator in $\L^2(O)^m$ associated with $a$. Then $L$ is invertible, maximal accretive, and sectorial. We take a closer look on the properties of $L$ in Section~\ref{Subsec: elliptic system}. In particular, $L$ possesses a square root $L^{\frac{1}{2}}$. The question if $\dom(L^{\frac{1}{2}}) = \IW^{1,2}_D(O)$ with equivalent norms became famous as \emph{Kato's square root problem}, and could be answered in the affirmative, first on the whole space~\cite{Kato-Square-Root-Proof}, and later under suitable geometric requirements on open sets~\cite{AKM,Darmstadt-KatoMixedBoundary,BEH}.

Phrased differently, Kato's square root property asserts that $L^{\frac{1}{2}}$ is an isomorphism $\IW^{1,2}_D(O) \to \L^2(O)^m$. It is then a natural question if $L^{\frac{1}{2}}$ also extrapolates in a suitable range of $p$ to an isomorphism $\IW^{1,p}_D(O) \to \L^p(O)^m$, where $\IW^{1,p}_D(O) \coloneqq \W^{1,p}_D(O)^m$. In the case $O=\R^d$ and $m=1$, an optimal range of such $p$ was given by Auscher~\cite{Memoirs}. When $O$ is a bounded and interior thick domain, $D$ is Ahlfors--David regular (see Assumption~\ref{Ass: D}), $O$ satisfies the so-called \emph{weak Lipschitz condition around $N$}, and the coefficients are real and scalar, then a first result for mixed boundary conditions was given by Auscher, Badr, Haller-Dintelmann, and Rehberg~\cite{ABHR}. Under the same geometric assumptions, but with complex and matrix-valued coefficients, Egert showed in~\cite{E} that $L^{\frac{1}{2}}$ extrapolates to an isomorphism $\IW^{1,p}_D(O) \to \L^p(O)^m$ if $p\in (p_{-}(L), 2+\eps)$. Here, $p_{-}(L)$ is the infimum of $\cI(L)$, where
\begin{align}
	\cI(L) \coloneqq \{ p\in (1,\infty) \colon \{ \e^{-tL} \}_{t>0} \text{ is $\L^p$-bounded} \},
\end{align}
and $\eps>0$ depends on geometry, dimensions, and coefficient bounds. We will write $\pLB{p} \coloneqq \sup_{t>0} \| \e^{-tL} \|_{\L^p \to \L^p}$ whenever $p\in \cI(L)$. In the situation of~\cite{ABHR} one has, for instance, $p_{-}(L) = 1$, see~\cite[Prop.~4.6~(i)]{ABHR}.
The notion of $\L^p$-bounded families of operators is made precise in Definition~\ref{Def: boundedness and ODE}. Under \emph{extrapolation} we understand that $L^{\frac{1}{2}}$ and $L^{-\frac{1}{2}}$, initially defined on $\IW^{1,p}_D(O) \cap \IW^{1,2}_D(O)$ and $\L^p(O)^m \cap \L^2(O)^m$, respectively, extend by continuity to bounded operators $\IW^{1,p}_D(O) \to \L^p(O)^m$ and $\L^p(O)^m \to \IW^{1,p}_D(O)$. In this case, we say that $L^{\frac{1}{2}}$ is a \emph{$p$-isomorphism}. Similarly, we say that $L^{\frac{1}{2}}$ and $L^{-\frac{1}{2}}$ are \emph{$p$-bounded} and so on.

The geometric constellation in~\cite{ABHR,E} was primarily dictated by the available $\L^2$-theory for the Kato square root problem from~\cite{Darmstadt-KatoMixedBoundary,Laplace-Extrapolation}. Meanwhile, by a recent result of the author together with Egert and Haller-Dintelmann~\cite{BEH}, Kato's square root problem could be solved in the affirmative if $O$ is a possibly unbounded, non-connected, not interior thick open set which is \emph{locally uniform near $N$} (see Assumption~\ref{Ass: N}), and whose Dirichlet part $D$ is Ahlfors--David regular (see Assumption~\ref{Ass: D}). It is one aim of this article to extend the results of~\cite{E} to this setting. %

Also, it was shown in~\cite{E} that if $p<p_{-}(L)$, then $L^{\frac{1}{2}}$ is not a $p$-isomorphism, but sharpness for the endpoint $p=p_{-}(L)$ was not treated. Moreover, it was left as an open question in~\cite[p.5]{E} to characterize the optimal range of $p>2$ in the case of mixed boundary conditions.

In summary, it is the goal of this paper to establish the following improvements of the state of the art presented in~\cite{E}. In fact our article should be seen as a part II to~\cite{E}, which is also the motivation for our title.

\begin{enumerate}
	\item The square root isomorphism $L^{\frac{1}{2}}: \IW^{1,2}_D(O) \to \L^2(O)^m$ can be extrapolated away from $2$ even if $O$ is only supposed to be open, locally uniform near $N$, and with Ahlfors--David regular Dirichlet part $D$, \label{item: Q1}
	\item the interior of the optimal range in which $L^{\frac{1}{2}}$ is a $p$-isomorphism is $(p_{-}(L), \tilde q_{+}(L))$, \label{item: Q2}
	\item if $p_{-}(L)>1$, then $L^{\frac{1}{2}}$ is not a $p_{-}(L)$-isomorphism, similarly for $\tilde q_{+}(L) < \infty$. \label{item: Q3}
\end{enumerate}

Here, $\tilde q_{+}(L)$ is defined as follows: We know from the Lax--Milgram lemma that $\cL : \IW^{1,2}_D(O) \to \IW^{-1,2}_D(O)$ is an isomorphism, where we put $\IW^{-1,p}_D(O) \coloneqq (\IW^{1,p'}_D(O))^*$ for $p\in (1,\infty)$. Here, $p'$ is the H\"{o}lder-conjugate exponent to $p$. Similarly to the case of $L^{\frac{1}{2}}$, we say that $\cL$ is a $p$-isomorphism if $\cL^{-1}$ extends from  $\IW^{-1, p}_D(O) \cap \IW^{-1,2}_D(O)$ to a bounded operator $\IW^{-1,p}_D(O) \to \IW^{1,p}_D(O)$. Then define the set
\begin{align}
\label{Def: cJ}
	\cJ(L) \coloneqq \{ p\in (1,\infty) \colon \cL \text{ is a $p$-isomorphism} \},
\end{align}
and put $\tilde q_{+}(L) \coloneqq \sup \cJ(L)$. Moreover, for $q\in \cJ(L)$, put $\qLB{q} \coloneqq \| \cL^{-1} \|_{\IW^{-1,p}_D(O) \to \IW^{1,p}_D(O)}$. The quantities $\pLB{p}$ and $\qLB{q}$ will be useful to quantify dependence of implicit constants in Theorem~\ref{Thm: main result}.

\begin{remark}
\label{Rem: q tilde plus}
	Clearly, $\cL$ always extends from $\IW^{1,p}_D(O) \cap \IW^{1,2}_D(O)$ to a bounded operator $\IW^{1,p}_D(O) \to \IW^{-1,p}_D(O)$, and the operator norm of this extension is controlled by the coefficient bounds. If $p \in \cJ(L)$, then this extension is one-to-one and onto, and its inverse coincides on $\IW^{-1, p}_D(O) \cap \IW^{-1,2}_D(O)$ with the inverse provided by the Lax--Milgram lemma. The compatibility with the inverse operator in the case $p=2$ is the main advantage of our definition, as it enables us to invoke interpolation arguments.
\end{remark}

On $\R^d$, a characterization of the upper endpoint using $\tilde q_+(L)$ already appeared in~\cite[Cor.~4.24]{Memoirs}. However, the version back then contained a flaw that originates from compatibility issues (confer with Remark~\ref{Rem: q tilde plus}) due to the unbounded geometry. This was pointed out by the author of~\cite{Memoirs} and Egert in their recent monograph~\cite[p. 135]{AE}. They also resolve the issue using compatible Hodge decompositions. Our characterization does not use this detour and works on domains with mixed boundary conditions.

The precise formulation of our main theorem reads then as follows.

\begin{theorem}
\label{Thm: main result}
	Let $O\subseteq \R^d$ be open, and $D \subseteq \bd O$ be closed. Assume that $(O,D)$ satisfies Assumption~\ref{Ass: N} and Assumption~\ref{Ass: D}. Then the system $L$ satisfies the following:
	\begin{enumerate}
		\item If $p_{-}(L) < r < p < 2$, then $L^{\frac{1}{2}}$ is a $p$-isomorphism and implicit constants depend on $p$, $r$, $\pLB{r}$, dimensions, geometry, and coefficient bounds. If $2 < q < r < \tilde q_{+}(L)$, then $L^{\frac{1}{2}}$ is a $q$-isomorphism, and implicit constants depend on $q$, $r$, $\qLB{r}$, dimensions, geometry, and coefficient bounds, \label{Item: main result 1}
		\item if $p\in (1,\infty)$ and $L^{\frac{1}{2}}$ is a $p$-isomorphism, then $p_{-}(L) \leq p \leq \tilde q_{+}(L)$, \label{Item: main result 2}
		\item if $p_{-}(L) > 1$, then $L^{\frac{1}{2}}$ is not a $p_{-}(L)$-isomorphism, and if $\tilde q_{+}(L) < \infty$, then $L^{\frac{1}{2}}$ is not a $\tilde q_{+}(L)$-isomorphism. \label{Item: main result 3}
	\end{enumerate}
\end{theorem}

\begin{remark}
	\label{Rem: critical numbers}
	The number $p_{-}(L)$ is at most $2_*$, see~\cite[Thm.~1.6]{E}, and the number $\tilde q_{+}(L)$ is at least $2$ by the Lax--Milgram lemma. If $d \geq 3$, $p_{-}(L)$ can be improved to $2_* - \eps$, see again~\cite[Thm.~1.6]{E}, and $\tilde q_{+}(L)$ can be improved to $\tilde q_{+}(L) + \eps$ in any dimension, see Proposition~\ref{Prop: Sneiberg}.
\end{remark}

In the next section, we will introduce all necessary definitions and make precise our geometric assumptions. Section~\ref{Sec: ODE and extrapolation} provides preliminary results, including some $p$-bounds for $L^{-\frac{1}{2}}$ and the $\H^\infty$ calculus, most of them taken from~\cite{E} without proof. Therefore, we advise the reader to keep a copy of that article handy.
In Section~\ref{Sec: p small survey}, we review results on the $p$-bound for $L^\frac{1}{2}$ when $p<2$. They rely on a weak-type argument which is based on a Calder\'{o}n--Zygmund decomposition for Sobolev functions. Its proof relies on a new Hardy's inequality related to mixed boundary conditions presented in Section~\ref{Subsec: hardy}.
Finally, we prove Theorem~\ref{Thm: main result} in the Sections~\ref{Sec: proof main result}-\ref{Sec: Endpoint cases}. The respective sections correspond to items~\ref{Item: main result 1}-\ref{Item: main result 3} in the theorem.

\subsection*{Notation}

Write $\diam( \cdot )$ for the diameter of a set and $\dist(\cdot, \cdot)$ for the distance between two sets. We employ the shorthand notation $\dist_E(x) \coloneqq \dist(E, \{ x \})$. The $(d-1)$-dimensional Hausdorff measure is denoted by $\HM^{d-1}$. If $\varphi \in (0,\pi)$, then write $\Sec_\varphi$ for the (open) sector $\{ z \in \C \setminus \{ 0 \} \colon |\arg(z)| < \varphi \}$, and put $\Sec_0 \coloneqq (0,\infty)$. Also, write $\overline{\Sec}_\varphi \coloneqq \overline{\Sec_\varphi}$ for the closed sector. For $\Xi \subseteq \C$ open, denote the space of bounded and holomorphic functions on $\Xi$ by $\H^\infty(\Xi)$ and equip it with the supremum norm. Inductively, we introduce the shorthand notation $2^{[0]} \coloneqq 2$ and $2^{[j+1]} \coloneqq (2^{[j]})^*$ for iterated Sobolev exponents.

\section{Elliptic systems \& function spaces}
\label{Sec: Systems and function spaces}

We properly introduce the $(m\times m)$-elliptic system $L$ on $O\subseteq \R^d$, $d\geq 2$, from the introduction. To this end, we also need to define Sobolev spaces subject to mixed boundary conditions, and we discuss geometric properties of the pair $(O,D)$.

\subsection{Geometry}
\label{Sec: Geometry}

Let $O \subseteq \R^d$ be open and $D\subseteq \bd O$ be closed. Fix the pair $(O,D)$ for the rest of this article. In particular, implicit constants might depend on this choice of geometry, as well as on dimensions. Introduce the following set of geometric assumptions.

\begin{assumption}{N}
\label{Ass: N}
There are $\eps \in (0,1]$ and $\delta \in (0,\infty]$ such that with $N_\delta  \coloneqq \{z \in \R^d : \dist(z,N) < \delta \}$ one has the following properties.

\begin{enumerate}
	\item All points  $x,y \in O \cap N_\delta$ with $|x-y| < \delta$ can be joined in $O$ by an \emph{$\eps$-cigar with respect to $\bd O \cap N_\delta$}, that is to say, a rectifiable curve $\gamma \subseteq O$ of length $\ell(\gamma) \leq \eps^{-1}|x-y|$ such that
	\begin{align}
	\label{eq: eps-delta}
	\dist(z, \bd O \cap N_\delta) \geq \frac{\eps|z-x|\, |z-y|}{|x-y|} \qquad \mathrlap{(z \in \gamma).}
	\end{align}

	\item $O$ has \emph{positive radius near $N$}, that is, there exists $\lambda > 0$ such that all connected components $O'$ of $O$ with $\bd O' \cap N \neq \emptyset$ satisfy $\diam(O') \geq \lambda \delta$.
\end{enumerate}
\end{assumption}

\begin{assumption}{D}
\label{Ass: D}
There are constants $C,c>0$ such that
\begin{align}
\label{Eq: d-1}
	\forall x\in D, r\in (0,\diam(D)]\colon \quad c r^{d-1} \leq \HM^{d-1}(\B(x,r)\cap D) \leq C r^{d-1}.
\end{align}
\end{assumption}

\begin{assumption}{D'}
\label{Ass: D'}
There are constants $C,c>0$ such that
\begin{align}
\label{Eq: d-1}
	\forall x\in D, r\in (0,1]\colon \quad c r^{d-1} \leq \HM^{d-1}(\B(x,r)\cap D) \leq C r^{d-1}.
\end{align}
\end{assumption}

Assumption~\ref{Ass: N} means that $O$ is \emph{locally uniform near~$N$}, see~\cite[Sec.~2.1]{BEH} for further information and a comparison with other geometric frameworks, including that used in~\cite{ABHR,E}. Assumption~\ref{Ass: D} means that $D$ is \emph{Ahlfors--David regular} and Assumption~\ref{Ass: D'} that $D$ is $(d-1)$-regular. If $D$ is bounded, then Assumption~\ref{Ass: D} and Assumption~\ref{Ass: D'} are equivalent, and in the unbounded case, Assumption~\ref{Ass: D} implies Assumption~\ref{Ass: D'}, but the converse might be false.

Throughout, assume that Assumption~\ref{Ass: N} is in place.

\subsection{Sobolev spaces}
\label{Subsec: Sobolev spaces}

Let $p\in [1,\infty)$. Write $\W^{1,p}(O)$ for the usual first-order Sobolev space of $\L^p$-functions on $O$ whose distributional gradient is in $\L^p(O)^d$. Also, introduce the set of test functions
\begin{align}
\label{Eq: C_D^infty}
	\Cont_D^\infty(O) \coloneqq \{ u|_O \colon u\in \Cont_0^\infty(\R^d) \quad \& \quad \dist(\supp(u), D) > 0 \},
\end{align}
where $\Cont_0^\infty(\R^d)$ denotes the set of smooth and compactly supported functions on $\R^d$.
Then the closed subspace $\W^{1,p}_D(O)$ of $\W^{1,p}(O)$ corresponding to \emph{mixed boundary conditions on $D$} is given as the closure of $\Cont_D^\infty(O)$ in $\W^{1,p}(O)$.
See~\cite[Sec.~2.2]{BHT} for an alternative set of test functions that leads to the same closure.
Moreover, define the space $\W^{-1,p}_D(O)$ by $\W^{1,p'}_D(O)^*$. Here, $X^*$ denotes the space of antilinear functionals on $X$, and the $\W^{-1,p}_D(O)$--$\W^{1,p'}_D(O)$ duality extends the $\L^2(O)$-duality. Replacing $\W$ by $\IW$, we extend all definitions to $m$-fold product spaces as before. All definitions can be extended to the more general case that $\Xi\subseteq \R^d$ is open and $E\subseteq \cl{\Xi}$; Then we obtain for instance $\W^{1,p}_E(\Xi)$.

Often, we will need that the space $\W^{1,p}_D(O)$ has the following (inhomogeneous) \emph{extension property}. This is the main result in~\cite{BHT}.
\begin{proposition}
\label{Prop: extension}
	Suppose Assumption~\ref{Ass: N}. Then there exists a linear extension operator $\Ext$ on $\L^1_\mathrm{loc}(O)$ that restricts for any $p\in (1,\infty)$ to a bounded operator $\Ext: \W^{1,p}_D(O) \to \W^{1,p}_D(\R^d)$.
\end{proposition}
Here, extension operator refers to the property that $(\Ext u)|_O = u$ for any $u\in \L^1_\mathrm{loc}(O)$. A consequence of Proposition~\ref{Prop: extension} is that classical inequalities like the Sobolev embedding $\W^{1,p}(\R^d) \subseteq \L^{p^*}(\R^d)$ valid for $p<d$ translate to $\W^{1,p}_D(O)$. Here, $p^* \coloneqq \frac{pd}{d-p}$ is the (upper) Sobolev exponent to $p$ if $p<d$, otherwise put $p^* = \infty$. Similarly, define the lower Sobolev exponent by $p_* \coloneqq \frac{pd}{d+p}$.

\subsection{The elliptic system}
\label{Subsec: elliptic system}

We give a precise definition for the elliptic system $L$ from the introduction. Consider the coefficients $a_{ij}, b_i, c_j, d \colon O \to \cL(\C^m)$. Here, $i$ and $j$ refer to row and column notation and $m\geq 1$ is the size of the system. Put $A=(a_{ij})_{ij}$, $b=(b_i)_i$, $c=(c_j)_j$. We assume the upper bound
\begin{align}
	\left\| \begin{bmatrix} d & c \\ b & A \end{bmatrix} \right\|_{\L^\infty(O;\cL(\C^{dm+m}))} \leq \Lambda
\end{align}
for the coefficients. Using the spaces from Section~\ref{Subsec: Sobolev spaces}, define the sesquilinear form
\begin{align}
	a \colon \IW^{1,2}_D(O) \times \IW^{1,2}_D(O) \to \C, \quad a(u,v) = \int_O \; \begin{bmatrix} d & c \\ b & A \end{bmatrix} \begin{bmatrix} u \\ \nabla u \end{bmatrix}  \cdot \cl{\begin{bmatrix} v \\ \nabla v \end{bmatrix}} \d x.
\end{align}
To ensure ellipticity of $a$, assume for some $\lambda > 0$ the (inhomogeneous) \emph{G\rr{a}rding inequality}
\begin{align}
\label{label-Garding}
\Re a(u,u) \geq  \lambda (\|u\|_2^2 + \|\nabla u\|_2^2) \qquad \mathrlap{(u \in \IW^{1,2}_D(O)).}
\end{align}
Associate with $a$ the operator
\begin{align}
\label{Def: cL}
	\cL \colon \IW^{1,2}_D(O) \to \IW^{-1,2}_D(O), \qquad \langle \cL u, v \rangle = a(u,v).
\end{align}
In virtue of~\eqref{label-Garding} and the Lax--Milgram lemma, $\cL$ is invertible. Define $L$ to be the maximal restriction of $\cL$ in $\L^2(O)^m$ in virtue of the inclusion $\L^2(O)^m \subseteq \IW^{-1,2}_D(O)$. Clearly, $L$ is again invertible. By ellipticity it follows that $L$ is densely defined. Moreover, for some $\omega'\in [0,\pi/2)$ that depends on coefficient bounds, the numerical range $\Theta(L)$ is contained in the closed sector $\cl{\Sec}_{\omega'}$. This is a consequence of $\Theta(L) \subseteq \Theta(a)$ together with ellipticity of $a$. Define $\omega$ as the infimum over all such $\omega'$. In particular, $L$ is sectorial and maximal accretive, hence generates a contraction semigroup on $\L^2(O)^m$.

We will freely use the sectorial functional calculus of $L$ and assume that the reader is familiar with this concept. The reader can consult the monograph~\cite{Haase} for further information on this topic.

The following theorem is the main result from~\cite{BEH} and establishes Kato's square root property for $L$.

\begin{theorem}[Kato's square root property]
\label{Thm: Kato}
	Assume that the pair $(O,D)$ satisfies Assumption~\ref{Ass: N} and Assumption~\ref{Ass: D}. Then $\dom(L^{\frac{1}{2}})=\IW^{1,2}_D(O)$ holds with equivalence of norms
	\begin{align}
	\label{Eq: Kato estimate}
		\| u \|_2 + \| \nabla u \|_2 \approx \| L^{\frac{1}{2}} u \|_2 \quad (u\in \IW^{1,2}_D(O)),
	\end{align}
	where the implicit constants depend only on geometry, dimensions, $\lambda$, and $\Lambda$.
\end{theorem}

\subsection{Decomposition of $\cL^{-1}$}
\label{Subsec: Decomposition cL}

We use a representation formula for $\cL^{-1}$. A similar formula, but for $\cL$ itself, was also used in~\cite{E}. In fact, the lemma below follows from the result in~\cite{E} and Theorem~\ref{Thm: Kato}. For convenience of the reader, we include the short and direct proof. The reason why we need a formula for $\cL^{-1}$ is that in unbounded sets the question of compatibility becomes a non-trivial task.

\begin{lemma}
\label{Lem: cL inverse repr}
	One has the identity
	\begin{align}
		(L^\frac{1}{2} \cL^{-1} u \SP f)_2 = \langle u, (L^*)^{-\frac{1}{2}} f \rangle \qquad (u \in \IW^{-1,2}_D(O), f\in \L^2(O)^m).
	\end{align}
\end{lemma}

\begin{proof}
	First, let $u\in \L^2(O)^m$ and $v \in \IW^{1,2}_D(O)$. Then $\cL^{-1} u \in \dom(L)$, hence, taking Theorem~\ref{Thm: Kato} into account, deduce
	\begin{align}
		(u \SP v)_2 = a(\cL^{-1} u, v) = (L \cL^{-1} u \SP v) = (L^\frac{1}{2} \cL^{-1} u \SP (L^*)^\frac{1}{2} v).
	\end{align}
	We specialize to $v \coloneqq (L^*)^{-\frac{1}{2}} f$ to conclude
	\begin{align}
		\langle u, (L^*)^{-\frac{1}{2}} f \rangle = (u \SP (L^*)^{-\frac{1}{2}} f)_2 = (L^\frac{1}{2} \cL^{-1} u \SP f).
	\end{align}
	Owing to Theorem~\ref{Thm: Kato} (applied with $L$ and $L^*$) and the Lax--Milgram lemma, both sides are continuous in $u$ with respect to the $\IW^{-1,2}_D(O)$ topology. Hence, the claim follows by density.
\end{proof}

\section{Review on off-diagonal estimates and $\L^p$ extrapolation}
\label{Sec: ODE and extrapolation}

\subsection{Off-diagonal estimates}
\label{Subsec: ODE}

We review decay properties in $\L^p$ of operator families related to $L$. Definition~\ref{Def: boundedness and ODE} also clarifies the notion that the family $\{ \e^{-tL} \}_{t>0}$ is $\L^p$-bounded, which was used in the definition of $p_{-}(L)$ in the introduction. The results obtained in this section will be used frequently in the course of this article.

\begin{definition}
\label{Def: boundedness and ODE}
	Let $\Xi \subseteq \R^d$ be measurable, $m_1,m_2$ natural numbers and let $U\subseteq \C \setminus \{0\}$ and $\cT = \{ T(z) \}_{z\in U}$ be a family of bounded operators $\L^2(\Xi)^{m_1} \to \L^2(\Xi)^{m_2}$. Given $1\leq p \leq q \leq \infty$, say that $\cT$ is \emph{$\L^p\to \L^q$ bounded} if there exists a constant $C>0$ such that for all $u\in \L^p(\Xi)^{m_1}\cap \L^2(\Xi)^{m_1}$ and $z\in U$ one has
	\begin{align}
	\label{Eq: Def Lp Lq boundedness}
		\| T(z)u \|_{\L^q(\Xi)^{m_2}} \leq C |z|^{-\frac{d}{2}\bigl(\frac{1}{p}-\frac{1}{q}\bigr)} \| u \|_{\L^p(\Xi)^{m_1}}.
	\end{align}
	If in addition there is $c\in (0,\infty)$ such that, whenever $E,F\subseteq \Xi$ and $\supp(u) \subseteq E$, the more restrictive estimate
	\begin{align}
	\label{Eq: Def Lp Lq off-diagonal}
		\| T(z)u \|_{\L^q(F)^{m_2}} \leq C |z|^{-\frac{d}{2}\bigl(\frac{1}{p}-\frac{1}{q}\bigr)} \e^{-c \frac{\dist(E,F)^2}{|z|}} \| u \|_{\L^p(E)^{m_1}}
	\end{align}
	holds, then say that $\cT$ satisfies \emph{$\L^p\to \L^q$ off-diagonal estimates}. Finally, if $p=q$ in the above situations, we simply talk about \emph{$\L^p$-boundedness} and \emph{$\L^p$ off-diagonal estimates}.
\end{definition}

If $\cT = \{ \e^{-tL} \}_{t>0}$ and $\cT$ is $\L^p \to \L^q$ bounded for \emph{some} values $1\leq p < q\leq \infty$, then we say that the semigroup is \emph{hypercontractive}.

Using Davies' trick one can show $\L^2$ off-diagonal estimates without \emph{any} geometrical requirements\footnote{In fact, using test functions as in~\eqref{Eq: C_D^infty} that are restrictions from $\R^d$ constitutes a form of \enquote{geometry}. To see that one can use test function classes that only use information on $O$ one has to impose restrictions on the geometry, see the discussion in Section~\ref{Subsec: Sobolev spaces}.}. The general argument is well-known,
see~\cite{Memoirs} for a version on $\R^d$ and~\cite{E} for a version on open sets. The following formulation follows with the proof presented in~\cite{E} when using Young's inequality also for the term of order zero (this also eliminates the dependence of the implicit constants on $\diam(O)$ that appeared in~\cite{E}) and when replacing the function $\varphi$ appearing in the proof by functions $\varphi_n \coloneqq \dist_E(x) \wedge n$ and taking the limit $n\to \infty$ in the end.

\begin{proposition}
\label{Prop: Gaffney estimates}
	For $\psi\in [0, \nicefrac{\pi}{2}-\omega)$, the operator families $\{ \e^{-zL} \}_{z\in \Sec_\psi}$, $\{ z \nabla \e^{-z^2 L} \}_{z\in \Sec_\psi}$, and $\{ z L \e^{-zL} \}_{z\in \Sec_\psi}$ satisfy $\L^2$ off-diagonal estimates, and the implied constants depend on $L$ only via its coefficient bounds.
\end{proposition}

We continue with estimates in $\L^p$. The following result allows to translate $\L^p$-boundedness into $\L^q\to \L^2$ off-diagonal estimates up to a slight loss in the integrability parameter. Its proof can be obtained by concatenating the relevant parts in~\cite[Prop.~4.4]{E}. Note that geometry in that result is only needed to have an extension operator at ones disposal. This is ensured by Proposition~\ref{Prop: extension} in our case.

\begin{proposition}[Off-diagonal estimates from boundedness]
\label{Prop: bdd to ODE}
	Let $q \in (p_{-}(L), 2)$, $p\in (q,2)$, and let $\psi \in [0, \nicefrac{\pi}{2}-\omega)$. Then $\{ \e^{-zL} \}_{z\in \Sec_\psi}$ satisfies $\L^p \to \L^2$ off-diagonal estimates, and the implicit constants depend on $p$, $q$, $\pLB{q}$, $\psi$, and coefficient bounds.
\end{proposition}

To the contrary, the following result yields $\L^p$-boundedness from hypercontractivity. The argument is similar to that in~\cite[Prop.~4.4]{E} and we only present the necessary changes.

\begin{proposition}[Boundedness from hypercontractivity]
\label{Prop: hypercontractive to bdd}
	Let $1\leq q < p < r\leq \infty$ be such that $\{ \e^{-tL} \}_{t>0}$ is $\L^q \to \L^r$ bounded. Then $\{ \e^{-tL} \}_{t>0}$ is $\L^p$-bounded, and $\pLB{p}$ depends only on $p$, $q$, $r$, coefficient bounds, and the implicit constant in the assumption.
\end{proposition}

\begin{proof}
	Let $\theta \in (0,1)$ be such that $\frac{1}{p} = \frac{1-\theta}{2} + \frac{\theta}{r}$ and define $[q,2]_\theta \in (q,2)$ by $\frac{1}{[q,2]_\theta} \coloneqq \frac{1-\theta}{q} + \frac{\theta}{2}$. Interpolate the $\L^q\to \L^r$ bounds from the assumption with the $\L^2$ off-diagonal bounds from Proposition~\ref{Prop: Gaffney estimates} to see that $\{ \e^{-tL} \}_{t>0}$ satisfies $\L^{[q,2]_\theta} \to \L^p$ off-diagonal estimates, where the implicit constants depend on $p$, $r$, and implied constants in the hypothesis. Then the claim follows from~\cite[Lem.~4.5]{E} with the same choices of $s$ and $g$ as in the proof~\cite[Prop.~4.4~(v)]{E}.
\end{proof}

As an application~\cite[Thm.~1.6]{E} one can derive upper bounds for $p_{-}(L)$ and lower bounds for $p_{+}(L)$. Geometry is only used to have an extension operator in hand.

\begin{corollary}
\label{Cor: size of semigroup interval}
	One has $p_{-}(L) = 1$ and $p_{+}(L) = \infty$ if $d=2$ and $p_{-}(L) \leq 2_*$ and $p_{+}(L) \geq 2^*$ if $d\geq 3$.
\end{corollary}

\subsection{Boundedness of the $\H^\infty$-calculus and Riesz transforms}
\label{Subsec: Hoo and Riesz}

In this section, we recall results from~\cite{Memoirs,E} on $\L^p$-boundedness of the $\H^\infty$-calculus and the Riesz transform associated with $L$. They are consequences of Proposition~\ref{Prop: bdd to ODE} and an extrapolation result due to Blunck and Kunstmann~\cite{Blunck-Kunstmann}. The use of geometry is completely hidden in the results providing off-diagonal estimates.

The following result is taken from~\cite[Thm.~1.3]{E}. Observe that the operator $f(L)$ is well-defined on $\L^2(O)^m$ owing to the Crouzeix--Delyon theorem for m-$\omega$-accretive operators~\cite[Cor.~7.1.17]{Haase}. Note that in the last line of~\cite[Thm.~1.3]{E}, the argument using the inclusion $\L^2(\Omega) \subseteq \L^p(\Omega)$ has to be substituted by a standard argument using Fatou's lemma.

\begin{proposition}[$\H^\infty$-calculus]
\label{Prop: Hoo}
	Let $p_{-}(L) < q < p_{+}(L)$, $p \in (q,2) \cup (2,q)$, and $\varphi \in (\omega, \pi)$. Then for every $f\in \H^\infty(\Sec_\varphi)$ one has
	\begin{align}
		\| f(L) u \|_p \lesssim \| f\|_\infty \|u \|_p \qquad (u\in \L^p(O)^m \cap \L^2(O)^m),
	\end{align}
	where the implicit constant depends on $p$, $q$, $\pLB{q}$, $\varphi$, and coefficient bounds.
\end{proposition}

In the same spirit, we obtain $\L^p$-boundedness of the Riesz transform, which upgrades to a $p$-bound for $L^{-\frac{1}{2}}$. The result is taken from~\cite[Lem.~6.1 \& Cor.~6.2]{E}.  As above, geometry is only used to provide off-diagonal estimates for $\{ t \nabla \e^{-t^2 L}\}_{t>0}$. For this it is crucial that $p<2$. Indeed, in this case the decomposition $\sqrt{2t} \nabla \e^{-2t L} = \sqrt{2} (\sqrt{t} \nabla \e^{-t L}) \e^{-t L}$ lets us conclude $\L^p \to \L^2$ off-diagonal estimates for $\{ t \nabla \e^{-t^2 L}\}_{t>0}$ from $\L^q$-boundedness of $\{ \e^{-t L} \}_{t>0}$ and $\L^2$ off-diagonal bounds for $\{ t \nabla \e^{-t^2 L}\}_{t>0}$ using composition, Proposition~\ref{Prop: bdd to ODE} and Proposition~\ref{Prop: Gaffney estimates}.

\begin{proposition}[Riesz transform]
\label{Prop: Riesz}
	Let $p_{-}(L) < q < p < 2$. Then the Riesz transform $\nabla L^{-\frac{1}{2}}$ is $\L^p$-bounded. Moreover, this bound can be upgraded to the $p$-bound
	\begin{align}
		\| L^{-\frac{1}{2}} u \|_{\IW^{1,p}(O)} \lesssim \| u \|_p \qquad (u\in \L^p(O)^m \cap \L^2(O)^m),
	\end{align}
	where implicit constant depends on $p$, $q$, $\pLB{q}$, and coefficient bounds.
\end{proposition}

\section[Survey on the p upper-bound when p < 2]{Survey on the $p$-bound for $L^\frac{1}{2}$ when $p<2$}
\label{Sec: p small survey}

In Proposition~\ref{Prop: Riesz} we have seen the $p$-bound for $L^{-\frac{1}{2}}$ when $p<2$. The goal of this section is to investigate the complementing $p$-bound for $L^\frac{1}{2}$. The argument is in large parts already known in the literature. Therefore, we will mainly review these known results here. Of course we will indicate all necessary modifications to adapt these results to our setting. There is, however, one ingredient needed in this section that is really novel: the global Hardy inequality adapted to an unbounded Dirichlet part in Theorem~\ref{Thm: Hardy inequality}. We will start with this result.

\subsection{Hardy's inequality}
\label{Subsec: hardy}

The main result of this subsection is the following Hardy's inequality.

\begin{theorem}[Hardy's inequality]
	\label{Thm: Hardy inequality}
	Assume that the pair $(O,D)$ satisfies Assumption~\ref{Ass: N} and that $D$ satisfies Assumption~\ref{Ass: D'}, and let $p\in (1,\infty)$. Then Hardy's inequality holds true for $\W^{1,p}_D(O)$, that is, for all $f\in \W^{1,p}_D(O)$ one has
	\begin{align}
		\int_O \left| \frac{f}{\dist_D} \right|^p \d x \lesssim \|f\|_{\W^{1,p}(O)}^p.
	\end{align}
\end{theorem}

Using the extension operator from Proposition~\ref{Prop: extension}, Theorem~\ref{Thm: Hardy inequality} is a direct consequence of the following whole-space version.

\begin{lemma}
	\label{Lem: Hardy whole space}
	Assume Assumption~\ref{Ass: D'}, and let $p\in (1,\infty)$. Then Hardy's inequality holds for $\W^{1,p}_D(\R^d)$, that is, for all $f\in \W^{1,p}_D(\R^d)$ holds
	\begin{align}
		\int_{\R^d} \left| \frac{f}{\dist_D} \right|^p \d x \lesssim \|f\|_{\W^{1,p}(\R^d)}^p.
	\end{align}
\end{lemma}

The proof of Lemma~\ref{Lem: Hardy whole space} relies on the following Hardy's inequality with pure Dirichlet boundary conditions, which is essentially contained in~\cite{Lehrback-PointwiseHardy}, see also~\cite{Hajlasz-PointwiseHardy}. Dependence of the implicit constants becomes apparent from an inspection of the proof.
\begin{proposition}
	\label{Prop: Dirichlet Hardy}
	Let $\Xi \subseteq \R^d$ be open. Assume that $\bd \Xi$ is Ahlfors--David regular, where either $\Xi$ is bounded or $\bd \Xi$ is unbounded. Then one gets the estimate
	\begin{align}
		\int_\Xi \left| \frac{f}{\dist_{\bd \Xi}} \right|^p \d x \lesssim \int_\Xi |\nabla f|^p \d x \mathrlap{\qquad (f\in \Cont_{\bd \Xi}^\infty(\Xi)).}
	\end{align}
	The implicit constant depends on geometry only via the implied constants from Ahlfors--David regularity of $\bd \Xi$.
	The inequality extends to $\W^{1,p}_{\bd \Xi}(\Xi)$ owing to Fatou's lemma.
\end{proposition}

Let us come back to the proof of Lemma~\ref{Lem: Hardy whole space}.

\begin{proof}[{Proof of Lemma~\ref{Lem: Hardy whole space}}]
	Let $(Q_k)_k$ be a grid of open cubes of diameter $\nicefrac{1}{4}$. We consider the sets $O_k \coloneqq 2 Q_k \setminus D$. Each $O_k$ has an Ahlfors--David regular boundary where the implicit constants depend only on the implied constants in Assumption~\ref{Ass: D'} and dimension.

	To see this, take a ball $B$ centered in $\bd O_k$ with radius $r$ at most $\nicefrac{1}{2}$ (which equals the diameter of $O_k$). One has $\bd O_k = \bd (2Q_k) \cup (D \cap 2 Q_k)$, which follows from porosity of $D$ (see~\cite[discussion before Cor.~2.11]{BEH}) and closedness of $D$ by elementary geometric arguments. Consequently, the lower bound follows from the $(d-1)$-regularity of $\bd(2 Q_k)$ or the $(d-1)$-regularity of $D$, depending on the location of the center of $B$. The upper bound follows similarly if $B$ doesn't intersect either $\bd (2 Q_k)$ or $D$. Otherwise, say $B$ is centered in $\bd (2 Q_k)$ and intersects $D$ in $x$. Then we estimate $\HM^{d-1}(B\cap \bd O_k) \leq \HM^{d-1}(B\cap \bd (2 Q_k)) + \HM^{d-1}(\B(x,2r)\cap D)$ and the estimate follows again from the $(d-1)$-regularity of the two portions of the boundary. Note that all constants are uniform in $k$.

	Now pick cutoff function $\chi_k$ which are supported in $2 Q_k$ and equal $1$ on $\cl{Q}_k$. Up to translation, we can use the same cut-off function for each $k$. Let $f\in \W^{1,p}_D(\R^d)$ and estimate using  Proposition~\ref{Prop: Dirichlet Hardy} and the bounded overlap of $(O_k)_k$ that
	\begin{align}
		\int_{\R^d\setminus D} \left| \frac{f}{\dist_D} \right|^p \d x \leq \sum_k \int_{O_k} \left| \frac{\chi_k f}{\dist_{\bd O_k}} \right|^p \d x  \lesssim \sum_k \|\chi_k f\|_{\W^{1,p}(2 Q_k)}^p \lesssim \|f\|_{\W^{1,p}(\R^d)}^p.
	\end{align}
	Note that at the first \enquote{$\lesssim$} we crucially use the control of implicit constants in the Dirichlet Hardy inequality.
\end{proof}

\subsection{Calder\'{o}n--Zygmund decomposition}
\label{Subsec: CZ}

The goal of this subsection is to investigate a Calder\'{o}n--Zygmund decomposition for functions in the Sobolev space $\IW^{1,p}_D(O)$. The reason for this is that the $p$-bound for $L^\frac{1}{2}$ in Corollary~\ref{Cor: square root upper bound} will follow from a weak-type estimate.
Such a decomposition was first shown by Auscher on the whole space~\cite[Appendix~A]{Memoirs}. This idea was refined in~\cite{ABHR,E} to work on domains, including the idea to use Hardy's inequality to include (partial) Dirichlet boundary conditions (and this is the reason why we have investigated Hardy's inequality in the previous subsection).
To formulate the precise result, we introduce $\C^m$-valued Sobolev spaces subject to different boundary conditions in the individual components.

\begin{definition}
	\label{Def: IW_D}
	Let $p\in [1,\infty)$, $\Xi \subseteq \R^d$ open and $E_k \subseteq \cl{\Xi}$ for $k=1,\dots,m$. With the array $\mathbb{E} \coloneqq (E_k)_{k=1}^m$ define the space
	\begin{align}
		\IW^{1,p}_{\mathbb{E}}(\Xi) \coloneqq \bigotimes_{k=1}^m \W^{1,p}_{E_k}(\Xi),
	\end{align}
	equipped with the subspace topology inherited from $\IW^{1,p}(\Xi)$. Moreover, introduce the abbreviation $\|\cdot \|_{\IW^{1,p}(\Xi)}$ for the norm on $\IW^{1,p}_{\mathbb{E}}(\Xi)$.
\end{definition}

\begin{remark}
	Here, we stay slightly more general than is necessary for our application. Compared to the $\L^2$ result used in~\cite{E}, we cannot deal with different Dirichlet boundary parts in different components in Theorem~\ref{Thm: Kato}. This is an artifact of the fact that $O$ might not be a doubling space.
\end{remark}

The main result of this subsection reads as follows.

\begin{theorem}[Sobolev Calder\'{o}n--Zygmund -- open set]
	\label{Thm: CZ}
	Let $O \subseteq \R^d$ be open, $D_k \subseteq \bd O$ be closed and $(d-1)$-regular for $k=1,\dots,m$, such that $O$ is a locally uniform domain near $\bd O \setminus D_k$ for all $k$, and let $1<p<\infty$. Then for every $u\in \IW^{1,p}_\ID(O)$ and every $\alpha > 0$ there exist an (at most) countable index set $J$, a family of cubes $(Q_j)_{j\in J}$, and functions $g,b_j \colon O\to \C^m$ for $j\in J$ such that the following holds.
	\begin{enumerate}
		\item $u=g+\sum_j b_j$ holds pointwise almost everywhere, \label{Item: CZ Sobolev decomposition}
		\item the family $(Q_j)_{j\in J}$ is locally finite, and every $x\in O$ is contained in at most $12^d$ cubes, \label{Item: CZ Sobolev overlap}
		\item $\sum_{j\in J} |Q_j| \lesssim \frac{1}{\alpha^p} \| u \|_{\IW^{1,p}(O)}^p$, \label{Item: CZ Sobolev cube measure sum}
		\item $g\in \Lip(O)^m$ with $\|g\|_{\Lip(O)^m} \lesssim \alpha$, \label{Item: CZ Sobolev good function}
		\item $b_j \in \IW^{1,p}_\ID(O)$ with $\| b_j \|_{\IW^{1,p}(O)} \lesssim \alpha |Q_j|^\frac{1}{p}$ for every $j\in J$, \label{Item: CZ Sobolev bad functions}
		\item if $p<d$, then $b_j \in \L^q(O)^m$ for $q\in [p,p^*]$ with $\| b_j \|_q \lesssim \alpha |Q_j|^{1/p+(1-\theta)/d}$, where $\theta\in [0,1]$ is such that $\nicefrac{1}{q} = \nicefrac{(1-\theta)}{p} + \nicefrac{\theta}{p^*}$, \label{Item: CZ Sobolev bad Lq estimate}
		\item $\|g\|_{\IW^{1,p}(O)} + \| \sum_{j\in J'} b_j \|_{\IW^{1,p}(O)} \lesssim \|u\|_{\IW^{1,p}(O)}$ for all $J' \subseteq J$, \label{Item: CZ Sobolev estimate against u}
		\item $b_j$ is supported in $Q_j\cap O$ for every $j$, \label{Item: CZ Sobolev bad localization}
		\item if $1<q<\infty$, $u\in \IW^{1,q}_\ID(O)$, and $J' \subseteq J$, then $\sum_{j\in J'} b_j$ converges unconditionally in $\IW^{1,q}_\ID(O)$. \label{Item: CZ Sobolev convergence}
	\end{enumerate}
\end{theorem}

Let us point out the differences with~\cite[Lem.~7.2]{E}. The most important difference is that we allow unbounded constellations in which $O$ is only locally uniform and $D$ does only fulfill a local $d-1$-dimensionality condition. Property~\ref{Item: CZ Sobolev bad Lq estimate} follows from an additional application of the Poincar\'{e} inequality and will be used for the case $d=2$ in Lemma~\ref{Lem: weak-type} later on. The bound for $J' \neq J$ in~\ref{Item: CZ Sobolev estimate against u} follows from an inspection of the proof, the same is true for the case $q\neq p$ in~\ref{Item: CZ Sobolev convergence}. The two last-mentioned properties are useful to circumvent convergence issues. These appear since we deal with operators that are a priori only bounded with respect to the $\L^2$ topology, and we cannot benefit from embedding relations as is the case of bounded domains.

In virtue of Proposition~\ref{Prop: extension}, Theorem~\ref{Thm: CZ} is an easy consequence of the following whole-space version. In particular, this shows that the only geometric ingredients are the availability of a Sobolev extension operator and Assumption~\ref{Ass: D'}.

\begin{lemma}[Sobolev Calder\'{o}n--Zygmund -- whole space]
	\label{Lem: CZ}
	Let $D_k \subseteq \R^d$ be closed and $(d-1)$-regular for $k=1,\dots,m$, and let $1<p<\infty$. For every $u\in \IW^{1,p}_\ID(\R^d)$ and every $\alpha>0$ there exist an (at most) countable index set $J$, a family of cubes $(Q_j)_{j\in J}$ and functions $g,b_j: \R^d \to \C^m$ for $j\in J$ such that the following holds.
	\begin{enumerate}
		\item $u=g+\sum_j b_j$ holds pointwise almost everywhere, \label{Item: CZ Sobolev Rd decomposition}
		\item the family $(Q_j)_j$ is locally finite, and every $x\in \R^d$ is contained in at most $12^d$ cubes, \label{Item: CZ Sobolev Rd overlap}
		\item $\sum_j |Q_j| \lesssim \frac{1}{\alpha^p} \|u\|_{\IW^{1,p}(\R^d)}^p$, \label{Item: CZ Sobolev Rd cube measure sum}
		\item $g\in \IW^{1,\infty}(\R^d)$ with $\|g\|_{\IW^{1,\infty}(\R^d)} \lesssim \alpha$, \label{Item: CZ Sobolev Rd good function}
		\item $b_j\in \IW^{1,p}_\ID(\R^d)$ with $\|b_j\|_{\IW^{1,p}(\R^d)} \lesssim \alpha |Q_j|^\frac{1}{p}$ for every $j$, \label{Item: CZ Sobolev Rd bad functions}
		\item $\|g\|_{\IW^{1,p}(\R^d)} + \| \sum_{j\in J'} b_j \|_{\IW^{1,p}(\R^d)} \lesssim \|u\|_{\IW^{1,p}(\R^d)}$ for all $J' \subseteq J$, \label{Item: CZ Sobolev Rd estimate against u}
		\item $b_j$ is compactly supported in $Q_j$ for every $j$, \label{Item: CZ Sobolev Rd bad localization}
		\item if $1<q<\infty$, $u\in \IW^{1,q}_\ID(\R^d)$ and $J' \subseteq J$, then $\sum_{j\in J'} b_j$ converges unconditionally in $\IW^{1,q}_\ID(\R^d)$. \label{Item: CZ Sobolev Rd convergence}
	\end{enumerate}
\end{lemma}

The proof is similar to those in~\cite[Lem.~7.2]{E} or~\cite[Lem.~7.1]{ABHR}. Indeed, in~\cite{E} they start with a function $u$ on $\Omega$ and decompose $U \coloneqq \eta (\Ext u)$, where $\Ext$ is an extension operator for $\Omega$ and $\eta$ is a cutoff function that is constantly $1$ on a ball $B$ that compactly contains $\Omega$. This way, they construct a global decomposition of $U$, but can rely on a Hardy's inequality on $B$ with Dirichlet boundary condition on $D \cup \bd B$. It is crucial for them that the auxiliary domain and Dirichlet boundary part are still bounded. That being said, the central insight to show Lemma~\ref{Lem: CZ} is to follow their lines of argument, but directly work with a global function $U$ and the global Hardy inequality established in Lemma~\ref{Lem: Hardy whole space}. We would like to point out that the order in which extension operator and Hardy's inequality are applied is reversed in the approach presented in this article.

\begin{remark}[boring cubes]
	We would likewise use the opportunity to mention that the so-called \emph{boring cubes} used in~\cite{E} are not needed. Indeed, they were introduced in~\cite{E} to treat Hardy-type estimates for bad functions directly in the construction. However, this bound can be deduced a posteriori using Hardy's inequality. With this in mind, bad functions on boring cubes should be defined like bad functions on \emph{usual cubes} and not like bad functions on \emph{special cubes}. Gradient and non-gradient estimates for bad functions on usual cubes then apply directly to bad functions on \enquote{boring cubes}. This shortens and conceptually simplifies the proof given in~\cite{E}.
\end{remark}

\subsection{Upper bound for the square root when $p<2$}
\label{Subsec: weak-type argument}

In the case $p<2$, we prove $p$-boundedness of the square root. The heart of the matter is the weak-type estimate in Lemma~\ref{Lem: weak-type}. A crucial observation is that we gain up to one Sobolev exponent in comparison to $p_{-}(L)$. We will benefit from this in Section~\ref{Sec: Endpoint cases} later on.

\begin{lemma}
\label{Lem: weak-type}
	Let $p_{-}(L)_* \vee 1 < q < p < 2$, then one has for all $\alpha>0$ the weak-type bound
	\begin{align}
	\label{Eq: weak type estimate}
		\bigl|\bigl\{ x\in O \colon |(L^{\frac{1}{2}} u)(x)| > \alpha \bigr\}\bigr| \lesssim \frac{1}{\alpha^p} \|u\|_{\IW^{1,p}(O)}^p \qquad (u\in \IW^{1,p}_D(O) \cap \IW^{1,2}_D(O)),
	\end{align}
	where implicit constants depend on $p$, coefficient bounds, and, if $p^* < 2$, on $q$ and $\pLB{q^*}$.
\end{lemma}

The proof of Lemma~\ref{Lem: weak-type} is similar to that of~\cite[Prop.~8.1]{E} and we will only explain the necessary changes, so the reader is strongly advised to keep a copy of that article handy. A key difference is that the argument in~\cite{E} assumes $p^* < 2$, which is not feasible in $d=2$.\footnote{After publication, the author of~\cite{E} resolved this problem in the arXiv version by a case distinction. We present a unified treatment here.} Property~\ref{Item: CZ Sobolev bad Lq estimate} in our Calder\'{o}n--Zygmund decomposition lets us circumvent this issue.

\begin{proof}
	To start with, we claim that there exists $r\in [p, p^*]$ satisfying
	\begin{enumerate}[label=(\alph*)]
		\item $\{ \e^{-tL} \}_{t>0}$ satisfies $\L^r \to \L^2$ off-diagonal estimates, \label{Item: weak-1}
		\item the $\H^\infty(\Sec_\varphi)$-calculus of $L$ is $\L^r$-bounded for all $\varphi\in (\omega,\pi)$, \label{Item: weak-2}
		\item $r\leq 2$. \label{Item: weak-3}
	\end{enumerate}

	Note that, of course,~\ref{Item: weak-3} is necessary for~\ref{Item: weak-1} to hold, but we will also need to use~\ref{Item: weak-3} in conjunction with~\ref{Item: weak-2}. Indeed, if $p^* \geq 2$, chose $r=2$, which is admissible by Proposition~\ref{Prop: Gaffney estimates} and the Crouzeix--Delyon theorem. Otherwise, when $p^*<2$, let
	\begin{align}
		p_{-}(L)_* \vee 1 < q < s_* < r_* < p < r.
	\end{align}
	Then, on the one hand, $p_{-}(L) < s < r$. On the other hand, by assumption of this case, $r < p^* < 2$. Hence, $\{ \e^{-tL} \}_{t>0}$ satisfies $\L^r \to \L^2$ off-diagonal estimates by Proposition~\ref{Prop: bdd to ODE}, and the $\H^\infty(\Sec_\varphi)$-calculus of $L$ is bounded on $\L^r$ for any $\varphi \in (\omega, \pi)$ according to Proposition~\ref{Prop: Hoo}. Since $\pLB{s} \leq \pLB{q^*}$ by interpolation with the contraction semigroup on $\L^2$, implied constants depend on $p$, $q$, $\pLB{q^*}$, $\varphi$, and coefficient bounds. Eventually, the dependence on $\varphi$ will be replaced by a dependence on $\omega$, which in turn is under control using the coefficients bounds. Note also that we have $r\in [p, p^*]$.

	Now let $\alpha > 0$, $u\in \IW^{1,p}_D(O)$, and let $u=g+\sum_j b_j$ be the Calder\'{o}n--Zygmund decomposition from Theorem~\ref{Thm: CZ}. Refer with~\ref{Item: CZ Sobolev Rd decomposition}-\ref{Item: CZ Sobolev Rd convergence} to the respective properties of the decomposition. As in~\cite{E}, the proof divides into 4 steps: Estimate of the good part, decomposition of the bad part into a local and a global integral, estimate of the local integral, and bound for the global integral.

	\textbf{Step 1}: Handling the good part. Decompose with the Calder\'{o}n--Zygmund decomposition
	\begin{align}
		\bigl|\bigl\{ |L^{\frac{1}{2}} u| > \alpha \bigr\}\bigr| \leq \bigl|\bigl\{ |L^{\frac{1}{2}} g| > \nicefrac{\alpha}{2} \bigr\}\bigr| + \bigl|\bigl\{ |L^{\frac{1}{2}} \sum_{j\in J} b_j| > \nicefrac{\alpha}{2} \bigr\}\bigr|.
	\end{align}
	We refer to the first term as the \emph{good part} and to the second term as the \emph{bad part}. With the exact same arguments as in~\cite[Step~1]{E}, we conclude
	\begin{align}
		\bigl|\bigl\{ |L^{\frac{1}{2}} g| > \nicefrac{\alpha}{2} \bigr\}\bigr| \lesssim \frac{1}{\alpha^p} \| u \|_{\IW^{1,p}(O)}^p.
	\end{align}

	\textbf{Step 2}: Decomposition of the bad part. First of all, let us mention that it suffices to assume that $J$ is finite, provided we can show a bound that does not depend on the size of $J$. This allows to rearrange terms without worrying about convergence issues. Indeed, put $J_n \coloneqq J \cap \{ 1, \dots, n \}$ for $n\geq 1$, then $\sum_{j\in J_n} b_j \to b = \sum_{j\in J} b_j$ in $\IW^{1,2}_D(O)$ as $n\to \infty$ by~\ref{Item: CZ Sobolev convergence}, so we get from Tchebychev's inequality
	\begin{align}
	\begin{split}
	\label{Eq: reduction to finite J}
		\bigl|\bigl\{ |L^{\frac{1}{2}} b| > \nicefrac{\alpha}{2} \bigr\}\bigr| &\leq \bigl|\bigl\{ |L^{\frac{1}{2}} (b-\sum_{j\in J_n} b_j)| > \nicefrac{\alpha}{4} \bigr\}\bigr| + \bigl|\bigl\{ |L^{\frac{1}{2}} \sum_{j\in J_n} b_j| > \nicefrac{\alpha}{4} \bigr\}\bigr| \\
		&\leq \frac{16}{\alpha^2} \bigl\| L^{\frac{1}{2}} \bigl(b-\sum_{j\in J_n} b_j \bigr) \bigr\|_2^2 + \bigl|\bigl\{ |L^{\frac{1}{2}} \sum_{j\in J_n} b_j| > \nicefrac{\alpha}{4} \bigr\}\bigr|.
	\end{split}
	\end{align}
	Now, if the second term can be controlled by
	\begin{align}
		\bigl|\bigl\{ |L^{\frac{1}{2}} \sum_{j\in J_n} b_j| > \nicefrac{\alpha}{4} \bigr\}\bigr| \lesssim \frac{1}{\alpha^p} \| u \|_{\IW^{1,2}(O)}^p
	\end{align}
	with an implicit constant independent of $n$, then, in the light of Theorem~\ref{Thm: Kato} and by convergence of $\sum_{j\in J_n} b_j$ to $b$ in $\IW^{1,2}_D(O)$, the first term vanishes as $n\to \infty$, which lets us conclude.

	Furthermore, to control $|\{ |L^{\frac{1}{2}} b| > \nicefrac{\alpha}{2} \}|$, it suffices as in~\cite[Step~2]{E} to control the \emph{local} and \emph{global integrals}
	\begin{align}
		\Bigl|\Bigl\{ \bigl|\sum_j \int_{2^{-n}}^{r_j \vee 2^{-n}} L \e^{-t^2 L} b_j \d t \bigr| > \frac{\sqrt{\pi} \alpha}{8} \Bigr\}\Bigr| + \Bigl|\Bigl\{ \bigl|\sum_j \int_{r_j \vee 2^{-n}}^\infty L \e^{-t^2 L} b_j \d t\bigr| > \frac{\sqrt{\pi} \alpha}{8} \Bigr\}\Bigr|
	\end{align}
	for all $n\geq 1$, where $r_j = 2^\ell$ with $\ell$ the unique integer such that $2^\ell \leq \ell_j \leq 2^{\ell+1}$, and where $\ell_j$ is the sidelength of $Q_j$.

	\textbf{Step 3}: Handling the local integral. Put $\gamma \coloneqq d(\nicefrac{1}{r}-\nicefrac{1}{2})+1$ and $C_k(Q_j) \coloneqq 2^{k+1} Q_j \setminus 2^k Q_j$ for $k \geq 2$. We only show the bound
	\begin{align}
	\label{Eq: weak step 3 crucial estimate}
		\Bigl\| \int_{2^{-n}}^{r_j \vee 2^{-n}} L\e^{-tL^2} b_j \d t \Bigr\|_{\L^2(C_k(Q_j)\cap O)^m} \lesssim \alpha \ell_j^{d/2} 2^{-k\gamma} \e^{-c4^k}
	\end{align}
	for all $k\geq 2$, $j\in J$, and $n\geq 1$. Then the bound for the local integral can be concluded as in~\cite[Step~3]{E}.

	So, let us show~\eqref{Eq: weak step 3 crucial estimate}. Clearly, we can assume $r_j \geq 2^{-n}$. Note that the off-diagonal bounds for $\{ \e^{-tL} \}_{t>0}$ can be upgraded to off-diagonal bounds for $\{ tL \e^{-tL} \}_{t>0}$ by composition (see for example~\cite[Prop.~4.4~(iv)]{E}), and this imports no further dependence for the implied constants. With this in hand, and using the support property~\ref{Item: CZ Sobolev Rd bad localization}, we calculate for the integrand of~\eqref{Eq: weak step 3 crucial estimate} that
	\begin{align}
		\| L \e^{-t^2L} b_j \|_{\L^2(C_k(Q_j)\cap O)^m} \lesssim t^{-d\bigl(\frac{1}{r}-\frac{1}{2}\bigr)-2} \e^{-c 4^{k-1} r_j^2 / t^2} \| b_j \|_r.
	\end{align}
	Plugging this back into~\eqref{Eq: weak step 3 crucial estimate} and using~\ref{Item: CZ Sobolev bad Lq estimate} leads to
	\begin{align}
		\Bigl\| \int_{2^{-n}}^{r_j} L\e^{-t^2L} b_j \d t \Bigr\|_{\L^2(C_k(Q_j)\cap O)^m} &\leq \int_{2^{-n}}^{r_j} \| L\e^{-t^2L} b_j \|_{\L^2(C_k(Q_j)\cap O)^m} \d t \\
		&\lesssim \alpha \ell_j^{\frac{d}{p} + 1-\theta} \int_{2^{-n}}^{r_j} t^{-d \bigl( \frac{1}{r} - \frac{1}{2} \bigr) - 1} \e^{-c 4^{k-1} r_j^2 / t^2} \dd{t}.
		\intertext{Using the substitution $s=4^{k-1} r_j^2 / t^2$, we obtain}
		&\lesssim \alpha \ell_j^{\frac{d}{p} + 1-\theta} r_j^{-\gamma} 2^{-(k-1) \gamma} \int_{4^{k-1}}^\infty s^{\frac{\gamma}{2}} \e^{-cs} \dd{s}.
		\intertext{Keeping in mind $\ell_j \approx r_j$ and observing $\theta=d\bigl(\nicefrac{1}{p}-\nicefrac{1}{r}\bigr)$, the prefactor reduces to $\ell_j^\frac{d}{2}$. Then we split the exponential term and use $s\geq 4^{k-1}$ to find}
		&\lesssim \alpha \ell_j^\frac{d}{2} 2^{-(k-1) \gamma} \e^{-c 4^{k-2}} \int_0^\infty s^{\frac{\gamma}{2}} \e^{-\nicefrac{cs}{2}} \dd{s}.
	\end{align}
	The integral in $s$ is finite since $\gamma\geq 1>0$. This completes the proof of~\eqref{Eq: weak step 3 crucial estimate}.

	\textbf{Step 4}: Estimate for the global integral. To this end, put $J_k \coloneqq \{ j\in J \colon r_j \vee 2^{-n} = 2^k \}$ for any integer $k$. Then, as in~\cite[Step~4]{E}, there exists a function $f$ which belongs to $\H^\infty(\Sec_\varphi)$ for any $\varphi\in (\omega,\nicefrac{\pi}{2})$, with which we can write
	\begin{align}
		\sum_{j\in J} \int_{r_j}^\infty L \e^{-t^2 L} b_j \d t = \sum_k \sum_{j\in J_k} \frac{1}{2^k} f(4^k L) b_j.
	\end{align}
	Plug this back into the definition of the global integral and use Tchebychev's inequality (compared to~\cite{E}, we apply it with $r$ instead of $p^*$), followed by linearity (here, we use that $J$ is supposed to be finite) to derive
	\begin{align}
		\Bigl|\Bigl\{ \bigl|\sum_j \int_{r_j}^\infty L \e^{-t^2 L} b_j \d t \bigr| > \frac{\sqrt{\pi} \alpha}{8} \Bigr\}\Bigr| &\lesssim \frac{1}{\alpha^r} \Bigl\| \sum_k \sum_{j\in J_k} 2^{-k} f(4^k L) b_j \Bigr\|_r^r \\
		&= \frac{1}{\alpha^r} \Bigl\| \sum_k  f(4^k L) \sum_{j\in J_k} 2^{-k} b_j \Bigr\|_r^r.
	\end{align}
	Now, use the square function estimate from~\cite[Lem.~8.2]{E} with $\varphi \in (\omega, \nicefrac{\pi}{2})$, which is justified by~\ref{Item: weak-2}, to give
	\begin{align}
		&\Bigl\| \sum_k  f(4^k L) \sum_{j\in J_k} 2^{-k} b_j \Bigr\|_r^r \lesssim \Bigl\| \Bigl( \sum_k  \bigl| \sum_{j\in J_k} 2^{-k} b_j \bigr|^2 \Bigr)^\frac{1}{2} \Bigr\|_r^r.
		\intertext{Write out the $\L^r$-norm and employ~\ref{Item: weak-3} to give}
		={} &\int_O \Bigl( \sum_k  \bigl| \sum_{j\in J_k} 2^{-k} b_j \bigr|^2 \Bigr)^\frac{r}{2} \d x \leq \int_O \sum_k  \Bigl( \sum_{j\in J_k} |2^{-k} b_j| \Bigr)^r \d x.
		\intertext{Continuing as in~\cite{E}, we arrive at}
		\lesssim{} &\sum_{j\in J} \ell_j^{-r} \int_O |b_j|^r \d x.
	\end{align}
	Plug this back in the calculation and use~\ref{Item: CZ Sobolev bad Lq estimate} followed by~\ref{Item: CZ Sobolev Rd cube measure sum} to conclude (recall $\theta=d\bigl(\nicefrac{1}{p}-\nicefrac{1}{r}\bigr)$)
	\begin{align}
		\Bigl|\Bigl\{ \bigl|\sum_{j\in J} \int_{r_j}^\infty L \e^{-t^2 L} b_j \d t \bigr| > \frac{\sqrt{\pi} \alpha}{8} \Bigr\}\Bigr| \lesssim \sum_{j\in J} |Q_j|^{-\frac{r}{d} + \frac{r}{p} + (1-\theta)\frac{r}{d}} = \sum_{j\in J} |Q_j| \lesssim \alpha^{-p} \| u \|_{\IW^{1,p}(O)}^p.
	\end{align}
	Combining the bounds for the good part, the local integral and the global integral, this gives~\eqref{Eq: weak type estimate}.
\end{proof}

Using interpolation, we conclude a $p$-bound for $L^\frac{1}{2}$. The interpolation theory used in~\cite{ABHR,E} does not apply as $O$ might not be locally doubling.

\begin{corollary}
\label{Cor: square root upper bound}
	Let $p_{-}(L)_* \vee 1 < q < p < 2$, then
	\begin{align}
		\| L^{\frac{1}{2}} u \|_p \lesssim \| u \|_{\IW^{1,p}(O)} \qquad (u\in \IW^{1,p}_D(O) \cap \IW^{1,2}_D(O)),
	\end{align}
	where the implicit constant depends on $p$, $q$, $\omega$, coefficient bounds, and, if $p^* < 2$, on $\pLB{q^*}$.
\end{corollary}

\begin{proof}
	Put $r \coloneqq \nicefrac{(p+q)}{2}$. Write $\L^{r,\infty}(O)$ for the usual weak $\L^r$-space on $O$. Recall that $\L^{r,\infty}(O)$ is a complete quasi-normed space. Owing to Lemma~\ref{Lem: weak-type}, $L^{\frac{1}{2}}$ extrapolates from $\IW^{1,r}_D(O)\cap \IW^{1,2}_D(O)$ to a bounded operator $$L^{\frac{1}{2}} \colon \IW^{1,r}_D(O) \to \L^{r,\infty}(O)^m.$$ Here, the implied constant depends on $p$, $q$, $\omega$, coefficient bounds, and, if $p^* < 2$, on $\pLB{q^*}$. Moreover, by Theorem~\ref{Thm: Kato}, we have
	\begin{align}
		L^{\frac{1}{2}} \colon \IW^{1,2}_D(O) \to \L^2(O)^m.
	\end{align}
	Chose $\theta\in (0,1)$ such that $\nicefrac{1}{p} = \nicefrac{(1-\theta)}{r} + \nicefrac{\theta}{2}$. Real interpolation yields an extension
	\begin{align}
		L^{\frac{1}{2}} \colon \bigl(\IW^{1,r}_D(O), \IW^{1,2}_D(O)\bigr)_{\theta,2} \to \bigl(\L^{r,\infty}(O)^m, \L^2(O)^m\bigr)_{\theta,2}.
	\end{align}
	It remains to determine the interpolation spaces. For this, it suffices to argue componentwise~\cite[Sec.~1.18.1]{Triebel}. This being said, the space on the right-hand side then coincides with $\L^p(O)^m$ according to~\cite[Sec.~1.18.6]{Triebel}. For the left-hand side, this follows from the whole space case in~\cite[Thm.~1.2]{IP} (here, we need Assumption~\ref{Ass: D'}) together with the retraction-coretraction principle~\cite[Sec.~1.2.4]{Triebel} applied with $R$ the pointwise restriction to $O$ and $S$ the extension operator from Proposition~\ref{Prop: extension}.
\end{proof}

\section{Extrapolation of the square root property}
\label{Sec: proof main result}

In this part, we prove Theorem~\ref{Thm: main result}~\ref{Item: main result 1}. Most of the work in the case $p<2$ has already been done in Sections~\ref{Subsec: Hoo and Riesz} and~\ref{Subsec: weak-type argument}, so we mainly focus on the case $p>2$.

Consider Figure~\ref{Fig: Decomposition L}. Lemmas~\ref{Lem: extrapolation L} and~\ref{Lem: extrapolation square root} make precise the following heuristic: If two out of three arrows in the diagram correspond to $p$-isomorphisms, then all arrows are $p$-isomorphisms.
\begin{figure}[H]
\centering
\[ \begin{tikzcd}
        \IW^{1,2}_D(O) \ar{dr}{L^{\frac{1}{2}}} \ar{dd}{\cL} & \\
        & L^2(O)^m \ar{dl}{L^{\frac{1}{2}}} \\
        \IW^{-1,2}_D(O) &
\end{tikzcd}\]
\caption{Decomposition of $\cL$ into two square roots}
\label{Fig: Decomposition L}
\end{figure}
If $\cL$ is a $p$-isomorphism, then Lemma~\ref{Lem: iso gives bdd semigroup} and the case $p<2$ yield that the lower $L^\frac{1}{2}$-arrow comes for free. This leads to the proof of Theorem~\ref{Thm: main result}~\ref{Item: main result 1} in the case $p>2$. In the same spirit, but using Lemma~\ref{Lem: necessary condition}, we are going to show necessity in the next section.

\begin{lemma}
\label{Lem: extrapolation L}
	Let $p\in (1,\infty)$ be such that $L^{\frac{1}{2}}$ is a $p$-isomorphism and $(L^*)^{\frac{1}{2}}$ is a $p'$-isomorphism. Then $\cL$ is a $p$-isomorphism, and implicit constants depend only on those in the assumptions.
\end{lemma}

A similar correspondence appeared in~\cite[Thm.~6.5]{DtER}. They work with realizations of $L$ in $\L^p$ that rely on the fact that on bounded domains $\L^q \subseteq \L^p$ if $q \geq p$. In particular, compatibility is never an issue. Our simple argument using compatible extensions extends also to the unbounded setting.

\begin{proof}
	To show the $p$-isomorphism property we only have to show that $\cL^{-1}$ is $\IW^{-1,p}_D(O) \to \IW^{1,p}_D(O)$ bounded on $\IW^{-1,p}_D(O) \cap \IW^{-1,2}_D(O)$ in the case of $\cL$, compare with Remark~\ref{Rem: q tilde plus}. Hence, let $u \in \IW^{-1,p}_D(O) \cap \IW^{-1,2}_D(O)$, and also let $h \in \L^{p'}(O)^m \cap \L^2(O)^m$. Calculate first using Lemma~\ref{Lem: cL inverse repr} and the $p'$-isomorphism property for $L^*$ that
	\begin{align}
		|(L^\frac{1}{2} \cL^{-1} u \SP h)| = |\langle u, (L^*)^{-\frac{1}{2}} h \rangle| \leq \| u \|_{\IW^{-1,p}_D(O)}  \| (L^*)^{-\frac{1}{2}} h \|_{\IW^{1,p'}_D(O)}  \lesssim \| u \|_{\IW^{-1,p}_D(O)} \| h \|_{p'}.
	\end{align}
	Taking the supremum over $h$ yields
	\begin{align}
	\label{Eq: intermediate p estimate}
		\| L^\frac{1}{2} \cL^{-1} u \|_p \lesssim \| u \|_{\IW^{-1,p}_D(O)}.
	\end{align}
	Implicit constants depend on the implied constants coming from the $p'$-isomorphism assumption. Write $\cL^{-1} = L^{-\frac{1}{2}} L^\frac{1}{2} \cL^{-1}$ and use the $p$-isomorphism hypothesis along with~\eqref{Eq: intermediate p estimate} to give
	\begin{align}
		\| \cL^{-1} u \|_{\IW^{1,p}_D(O)} = \| L^{-\frac{1}{2}} L^\frac{1}{2} \cL^{-1} u \|_{\IW^{1,p}_D(O)} \lesssim \| L^\frac{1}{2} \cL^{-1} u \|_p \lesssim \| u \|_{\IW^{-1,p}_D(O)}. &\qedhere
	\end{align}

\end{proof}

The proof of the following lemma is similar, so we omit its proof.

\begin{lemma}
\label{Lem: extrapolation square root}
	Let $p\in (1,\infty)$ be such that $\cL$ is a $p$-isomorphism and $(L^*)^{\frac{1}{2}}$ is a $p'$-isomorphism. Then $L^{\frac{1}{2}}$ is a $p$-isomorphism, and implicit constants depend only on those in the assumptions.
\end{lemma}

To verify in the proof of Theorem~\ref{Thm: main result}~\ref{Item: main result 1} that Lemma~\ref{Lem: extrapolation square root} is applicable, we will rely on the following lemma. Recall the notation $2^{[j]}$ for iterated Sobolev exponents.

\begin{lemma}
\label{Lem: iso gives bdd semigroup}
	Let $2 \leq p < r < \tilde q_{+}(L)^*$. Then $p\in \cI(L)$, and $\pLB{p}$ depends only on $p$, $r$, $\qLB{r_*}$, coefficient bounds, and dimensions. In particular, $p_{+}(L) \geq \tilde q_{+}(L)^*$.
\end{lemma}

\begin{proof}
	First of all, we can assume that $p\geq 2^*$ in the light of Corollary~\ref{Cor: size of semigroup interval}. Let $\ell \geq 1$ denote the largest integer such that $2^{[\ell]} \leq p < 2^{[\ell+1]}$. Moreover, we can assume that $r \in (p, 2^{[\ell+1]})$, since otherwise we can replace $r$ by some exponent $s$ in this interval, and $s$ and $\qLB{s}$ depend only on $p$, $r$, and $\qLB{r}$, in virtue of interpolation.

	Let $p_j$ and $r_j$, $j=0, \dots, \ell$, denote the sequences of numbers satisfying $p_\ell = p$, $r_\ell = r$, and $p_{j+1} = p_j^*$ as well as $r_{j+1} = r_j^*$ for $j=0, \dots, \ell - 1$.

	For $j=0, \dots, \ell - 1$, we claim the following. Assume that $p_j \in \cI(L)$. Then $p_{j+1} \in \cI(L)$ and $\pLB{p_{j+1}}$ depends only on $p_j$, $r_j$, $r$, $\pLB{p_j}$, $\qLB{r}$, coefficient bounds, and dimensions.

	Indeed, observe first that $(r_{j+1})_{**} = (r_j)_* < 2^{[j]} \leq p_j$, and that $(r_{j+1})_{**} \geq (r_1)_{**} > (p_1)_{**} \geq 2_*$. Hence, by Proposition~\ref{Prop: Hoo} (taking Corollary~\ref{Cor: size of semigroup interval} into account), for some $\varphi \in (0,\nicefrac{\pi}{2}-\omega)$ depending only on coefficient bounds and dimensions, the $\H^\infty(\Sec_\varphi)$-calculus of $L$ is bounded on $\L^{(r_{j+1})_{**}}$, with implied constant depending only on $r_j$, $p_j$, $\pLB{p_j}$, coefficient bounds, and dimensions. Second, note that $\cL^{-1}$ is $(r_{j+1})_*$-bounded due to $2 \leq (p_1)_* < (r_{j+1})_* \leq r_*$, the assumption on $r$, and interpolation. The $(r_{j+1})_*$-bound for $\cL^{-1}$ only imports $r$, and $\qLB{r_*}$ as a further dependence. With these two ingredients, estimate for $t>0$ and $u\in \L^p(O)^m \cap \L^2(O)^m$ that
	\begin{align}
	\begin{split}
	\label{Eq: hypercontractive in iteration argument}
		\| \e^{-tL} u \|_{r_{j+1}} &\lesssim \| \cL^{-1} \cL \e^{-tL} u \|_{\IW^{1,(r_{j+1})_*}_D(O)} \lesssim \| \cL \e^{-tL} u \|_{\IW^{-1,(r_{j+1})_*}_D(O)} \\
		&\lesssim t^{-1} \| t L \e^{-tL} u \|_{(r_{j+1})_{**}} \lesssim t^{-1} \| u \|_{(r_{j+1})_{**}}.
	\end{split}
	\end{align}
	Implicit constants depend only on the aforementioned quantities. Estimate~\eqref{Eq: hypercontractive in iteration argument} means that the family $\{ \e^{-tL} \}_{t>0}$ is $\L^{(r_{j+1})_{**}} \to \L^{r_{j+1}}$ bounded. Now, Proposition~\ref{Prop: hypercontractive to bdd} yields $p_{j+1} \in \cI(L)$ with the claimed control over $\pLB{p_{j+1}}$. This completes the proof of the claim.

	Now, $2 \leq p_0 < 2^*$, so according to Corollary~\ref{Cor: size of semigroup interval}, $p_0 \in \cI(L)$ with $\pLB{p_0}$ depending only on $p_0$, coefficient bounds, and dimensions. Hence, by induction over $j=0,\dots,\ell-1$, the claim yields $p=p_\ell\in \cI(L)$. By finiteness of the sequence, and since $p_j$ and $r_j$ only depend on $p$, $r$, and $j$, the quantity $\pLB{p}$ depends only on $p$, $r$, $\qLB{r_*}$, coefficient bounds and dimensions, as claimed.
\end{proof}

\begin{proof}[Proof of Theorem~\ref{Thm: main result}~\ref{Item: main result 1}]
	The case $p_{-}(L) < r < p < 2$ is immediate by Corollary~\ref{Cor: square root upper bound} and Proposition~\ref{Prop: Riesz}.

	Now, let $2 < q < r < \tilde q_{+}(L)$. By definition, $\cL$ is a $q$-isomorphism. Then Lemma~\ref{Lem: extrapolation square root} lets us conclude provided we can ensure that $(L^*)^{\frac{1}{2}}$ is a $q'$-isomorphism. To this end, we want to appeal to the first case above, but applied to $L^*$ instead of $L$. By duality, $p_{-}(L^*) < q' < 2$ if, and only if, $2< q < p_{+}(L)$. But this is true by Lemma~\ref{Lem: iso gives bdd semigroup}, since $\cL$ is moreover an $r$-isomorphism. Hence, we can indeed apply the case above. Note that $\qLB{q}$ is controlled by $\qLB{r}$ in virtue of interpolation, which gives the right constant dependencies.
\end{proof}

\section{Necessary conditions}
\label{Sec: necessary conditions}

We show Theorem~\ref{Thm: main result}~\ref{Item: main result 2}. For its proof we will need the following lemma, whose proof is similar to that of Lemma~\ref{Lem: iso gives bdd semigroup}.

\begin{lemma}
\label{Lem: necessary condition}
	Let $p\in (2,\infty)$ be such that $L^{\frac{1}{2}}$ is a $p_*$-isomorphism. Then $p_{+}(L) \geq p$.
\end{lemma}

\begin{proof}
	In the light of Corollary~\ref{Cor: size of semigroup interval}, it suffices to treat the case $d\geq 3$. We employ an iterative argument. To this end, we make for $k\geq 0$ the following

	\textbf{Claim}: If $p_{+}(L) \geq 2^{[k]}$, $2^{[k]} < q < 2^{[k+1]}$, and $L^{\frac{1}{2}}$ is a $q_*$-isomorphism, then $p_{+}(L) \geq q$.

	\begin{proof}[Proof of the claim]
		We use the expansion $\e^{-t^2L} = L^{-\frac{1}{2}} L^{\frac{1}{2}} \e^{-t^2L}$. Observe that $2_* < q_* < 2^{[k]} \leq p_{+}(L)$. Hence, owing to Proposition~\ref{Prop: Hoo} and taking Corollary~\ref{Cor: size of semigroup interval} into account, the $\H^\infty(\Sec_\varphi)$-calculus of $L$ is bounded on $\L^{q_*}(O)^m$ for any $\varphi \in (\omega, \nicefrac{\pi}{2})$. Using this in the last step, together with the Sobolev embedding and the $q_*$-isomorphism property of $L^{\frac{1}{2}}$, yields
		\begin{align}
			\| \e^{-t^2 L} u \|_q \lesssim  \| L^{-\frac{1}{2}} L^{\frac{1}{2}} \e^{-t^2 L} u \|_{\IW^{1,q_*}_D(O)} \lesssim \| L^{\frac{1}{2}} \e^{-t^2 L} u \|_{q_*} \lesssim t^{-1} \| u \|_{q_*}.
		\end{align}
		This means that $\{ \e^{-tL} \}_{t>0}$ is $\L^{q_*} \to \L^q$ bounded. We conclude $q\leq p_{+}(L)$ with Proposition~\ref{Prop: hypercontractive to bdd}.
	\end{proof}

	Now, fix $\ell\geq 0$ such that $2 < \dots < 2^{[\ell]} < p \leq 2^{[\ell+1]}$. Observe that $L^{\frac{1}{2}}$ is by interpolation a $(2^{[j]})_*$-isomorphism for $j=1,\dots,\ell$. Hence, the claim yields by induction over $j$ that $p_{+}(L) \geq 2^{[\ell]}$. In a second step, the claim applied with $k=\ell$ and $q=p$ gives the assertion.
\end{proof}

\begin{proof}[Proof of Theorem~\ref{Thm: main result}~\ref{Item: main result 2}]
	We divide the proof into two cases.

	\textbf{Case 1}: $p<2$. This part was already shown in~\cite[Thm.~1.2~(ii)]{E}. Note that geometry in-there was only used to ensure the continuous embedding $\IW^{1,q}_D(O) \subseteq \L^{q^*}(O)^m$, which is a consequence of Proposition~\ref{Prop: extension} in our situation.

	\textbf{Case 2}: $p>2$. Assume that $L^\frac{1}{2}$ is a $p$-isomorphism. From $p=(p^*)_*$ and Lemma~\ref{Lem: necessary condition} it follows that $p_{+}(L) \geq p^*$ if $p<d$ and $p_+(L) = \infty$ otherwise. In particular, $2 < p < p_{+}(L)$, which translates to $p_{-}(L^*) < p' < 2$ by the duality formula $(\e^{-tL})^* = \e^{-tL^*}$.
	Hence, we obtain from the case $p<2$ for $L^*$ in Theorem~\ref{Thm: main result}~\ref{Item: main result 1} that $(L^*)^{\frac{1}{2}}$ is a $p'$-isomorphism. Then we conclude with Lemma~\ref{Lem: extrapolation L}  that $\cL$ is a $p$-isomorphism. Consequently, $\tilde q_{+}(L) \geq p$ as desired.
\end{proof}

\section{Endpoint cases}
\label{Sec: Endpoint cases}

To conclude this article, we show sharpness at the endpoints, provided they do not fall outside the interval $(1, \infty)$. This is Theorem~\ref{Thm: main result}~\ref{Item: main result 3}.

The following result is an application of {\u{S}}ne{\u{\ii}}berg's theorem \cite{Sneiberg-Original}. Here, we exploit that $\{ \IW^{1,p}_D(O) \}_{p\in (1,\infty)}$, $\{ \L^p(O)^m \}_{p\in (1,\infty)}$, and $\{ \IW^{-1,p}_D(O) \}_{p\in (1,\infty)}$ are complex interpolation scales. For the second scale, this is clear. That the first scale is an interpolation scale was discussed in the proof of Theorem~\ref{Thm: main result}~\ref{Item: main result 1} (note that the arguments in-there work both for the real and complex interpolation method). Finally, use duality to transfer the interpolation properties of the first scale to the last scale, see also~\cite[Prop.~5.2]{IP}.

\begin{proposition}
\label{Prop: Sneiberg}
	Assume that Assumption~\ref{Ass: D'} holds. Let $p\in (1,\infty)$ be such that $L^{\frac{1}{2}}$ is a $p$-isomorphism and suppose that there exists some $\eps' > 0$ such that $L^{\frac{1}{2}}$ is $q$-bounded for $q\in [p-\eps', p+\eps']$. Then there is $\eps' > \eps>0$ such that $L^{\frac{1}{2}}$ is a $q$-isomorphism for all $q\in (p-\eps,p+\eps)$. Similarly, if $\cL$ is a $p$-isomorphism, then there is again some $\eps>0$ such that $\cL$ is a $q$-isomorphism for all $q\in (p-\eps,p+\eps)$.
\end{proposition}

\begin{remark}
	Observe that in Proposition~\ref{Prop: Sneiberg} we did not assume that $\cL$ is $q$-bounded in an interval around $p$. This is because $\cL$ is automatically $q$-bounded for all $q\in [1,\infty]$ by H\"{o}lder's inequality applied to the definition of $\cL$.
\end{remark}

\begin{proof}[Proof of Theorem~\ref{Thm: main result}~\ref{Item: main result 3}]
	To begin with, assume that $p_{-}(L) > 1$ and $L^{\frac{1}{2}}$ is a $p_{-}(L)$-isomorphism. By Corollary~\ref{Cor: square root upper bound}, $L^{\frac{1}{2}}$ is $q$-bounded for $q\in (p_{-}(L)_* \vee 1,2)$. Since $p_{-}(L)>1$ and $p_{-}(L)_* < p_{-}(L)$, we find $\eps'>0$ as required by Proposition~\ref{Prop: Sneiberg}. Then, let $\eps>0$ be provided by that proposition. It follows from Theorem~\ref{Thm: main result}~\ref{Item: main result 2} that $p_{-}(L) \leq p_{-}(L)-\eps$, which is a contradiction.

	Now assume that $\tilde q_{+}(L) < \infty$ and that $L^{\frac{1}{2}}$ is a $\tilde q_{+}(L)$-isomorphism. Then $p_{+}(L) \geq \tilde q_{+}(L)^* > q_{+}(L)$ by Lemma~\ref{Lem: necessary condition}. Duality reveals $p_{-}(L^*) < (\tilde q_{+}(L))'$. It follows from Theorem~\ref{Thm: main result}~\ref{Item: main result 1} that $(L^*)^\frac{1}{2}$ is a $(\tilde q_{+}(L))'$ isomorphism. Consequently, Lemma~\ref{Lem: extrapolation L} gives that $\cL$ is a $\tilde q_{+}(L)$-isomorphism. Then, as above, Proposition~\ref{Prop: Sneiberg} leads to the contradiction $\tilde q_{+}(L) + \eps \leq \tilde q_{+}(L)$.

\end{proof}

\section*{Acknowledgments}

The author thanks \enquote{Studienstiftung des deutschen Volkes} for financial and academic support. Furthermore, the author was partially supported by the ANR project RAGE: ANR-18-CE-0012-01. The author thanks Moritz Egert for valuable discussions around the topic, and the anonymous referee for suggestions regarding the presentation.

\end{document}